\documentclass[12pt]{amsart}
\usepackage{amssymb}

\newtheorem{thm}{Theorem}[section]
\newtheorem{THM}{Theorem}

\newtheorem{lemma}[thm]{Lemma}
\newtheorem{cor}[thm]{Corollary}
\newtheorem{claim}{Claim}[thm]
\newtheorem*{sclaim}{Subclaim}

\theoremstyle{remark}
\newtheorem*{notation}{Notation}
\newtheorem*{remark}{Remark}

\theoremstyle{definition}
\newtheorem{defn}[thm]{Definition}

\fboxsep0.1mm\def\sc{\framebox[3.4mm][l]{$\clubsuit$}\hspace{0.5mm}{}}
\DeclareMathOperator{\pr}{Pr}
\DeclareMathOperator{\rts}{{\bf rts}}
\DeclareMathOperator{\rtts}{{\bf t}}
\DeclareMathOperator{\Tr}{Tr}
\DeclareMathOperator{\tr}{Tr^\circ}
\DeclareMathOperator{\trm}{Tr_{\langle \textit{m} \rangle}^\circ}
\DeclareMathOperator{\trw}{Tr_{\langle \textit{w} \rangle}^\circ}
\DeclareMathOperator{\trs}{Tr_{\sigma}^{\circ}}
\DeclareMathOperator{\ch}{Ch}
\DeclareMathOperator{\pp}{pp}
\DeclareMathOperator{\sk}{sk}

\DeclareMathOperator{\dom}{dom}
\DeclareMathOperator{\rng}{rng}
\DeclareMathOperator{\gch}{\text{\sf GCH}}
\DeclareMathOperator{\zfc}{\text{\sf ZFC}}
\DeclareMathOperator{\lift}{Lift}
\DeclareMathOperator{\cf}{cf}

\def\s{\subseteq}
\def\bks{\setminus}

\begin{document}
\title[transforming rectangles into squares]{transforming rectangles into squares, with applications to strong colorings}
\author{Assaf Rinot}
\address{Center for advanced studies in Mathematics,
Ben-Gurion University of the Negev, P.O.B. 653,
Be'er Sheva 84105,
Israel.}
\email{paper13@rinot.com}
\urladdr{http://www.assafrinot.com}
\begin{abstract} It is proved that every singular cardinal  $\lambda$ admits a function $\rts:[\lambda^+]^2\rightarrow[\lambda^+]^2$
that transforms rectangles into squares. Namely, for every cofinal subsets  $A,B$ of $\lambda^+$, there exists a cofinal subset $C\s\lambda^+$,
such that $\rts[A\circledast B]\supseteq C\circledast C$.

When combined with a recent result of Eisworth, this shows that Shelah's notion of strong coloring $\pr_1(\lambda^+,\lambda^+,\lambda^+,\cf(\lambda))$
coincides with the classical negative partition relation $\lambda^+\not\rightarrow[\lambda^+]^{2}_{\lambda^+}$.
\end{abstract}
\maketitle

\section{Introduction}
One of the fundamental fields of research in infinite combinatorics is the study of partition relations.
To recall some results in this area, we shall need the following piece of notation.
For cardinals $\kappa$ and $\theta$, the (negative) partition relation $\kappa\not\rightarrow[\kappa]^2_\theta$
stands for the existence of a function $f:[\kappa]^2\rightarrow\theta$ such that
$f``[A]^2=\theta$ for every $A\s \kappa$ of size $\kappa$.

In a paper from 1929, Ramsey  proved 
that $\omega\not\rightarrow[\omega]^2_2$ \emph{fails},
that is, every symmetric function $f:[\omega]^2\rightarrow 2$ admits an infinite homogenous set.
A few years later, Sierpi\'nski 
proved that Ramsey's theorem does not generalize to higher cardinals, introducing
a function witnessing $\omega_1\not\rightarrow[\omega_1]^2_2$.
At the 1960's,
Erd\"os, Hajnal and Rado \cite{partitionrelationsforcardinalnumbers} proved that the Generalized Continuum Hypothesis ($\gch$) implies
$\lambda^+\not\rightarrow[\lambda^+]^{2}_{\lambda^+}$ for every infinite cardinal $\lambda$ (see also \cite[Theorem 49.1]{combintorialsettheory}),
and at the mid 1980's, Todorcevic \cite{todorcevic} made a breakthrough, proving outright in $\zfc$ that $\lambda^+\not\rightarrow[\lambda^+]^2_{\lambda^+}$
holds for every regular cardinal $\lambda$.

The question of whether $\lambda^+\not\rightarrow[\lambda^+]^2_{\lambda^+}$ may fail for a singular cardinal $\lambda$ is still open.
In parallel, throughout the years, a study of stronger notions of negative partition relations was carried (see, for instance, Shelah \cite{sh276},\cite{sh:g},\cite{sh572},
and Shelah-Eisworth \cite{sh535},\cite{EiSh:819}).
In this paper, we shall be interested in the following two particular strengthenings:
\begin{itemize}
\item (Rectangular form) $\kappa\not\rightarrow[\kappa;\kappa]^2_\theta$
asserts the existence of a function $f:[\kappa]^2\rightarrow\theta$ such that
$f[A\times B]=\theta$ for every $A\s\kappa,B\s\kappa$ of size $\kappa$.
\item (Shelah's strong coloring \cite{sh:g}) $\pr_1(\kappa,\kappa,\theta,\sigma)$ asserts the existence of a  function $f:[\kappa]^2\rightarrow\theta$ such that
for every $\mathcal A\subseteq\mathcal P(\kappa)$ of size $\kappa$, consisting of pairwise disjoint sets of size $<\sigma$, and every $\xi<\theta$,
there exists distinct $a,b\in\mathcal A$ such that $f[a\times b]=\{\xi\}$.
\end{itemize}

Comparing the above three concepts is, of course, a very natural task.
To appreciate this task,
we mention that while the following results goes back to the work of Todorcevic from the 1980's:
\begin{itemize}
\item $\omega_1\not\rightarrow[\omega_1]^2_{\omega_1}$;
\item $\pp(\lambda)=\lambda^+$
implies $\lambda^+\not\rightarrow[\lambda^+]^2_{\lambda^+}$ for every singular cardinal $\lambda$,
\end{itemize}
the stronger versions was proved rather recently:
\begin{itemize}
\item (Moore, 2006 \cite{lspace}) $\omega_1\not\rightarrow[\omega_1;\omega_1]^2_{\omega_1}$;
\item (Eisworth, 2009 \cite{eisworth2}) $\pp(\lambda)=\lambda^+$ implies $\pr_1(\lambda^+,\lambda^+,\lambda^+,\cf(\lambda))$
for every singular cardinal $\lambda$.
\end{itemize}

In \cite{rinot12}, the author introduced the \emph{Ostaszewski square} $\sc_\kappa$,
which is a generalization of Jensen's square principle $\square_\kappa$ from \cite{jensen1972}.
One of the application appearing in \cite{rinot12} addresses negative partition relations that are derived from Ostaszewski squares.
Of course, we were hoping to yield as stronger partition relations as possible,
and incidently, around the same time, Eisworth \cite{eisworth2} made a significant progress on the old question of comparing the different notions,
showing that some strong coloring statements
are indeed equivalent to seemingly weaker statements.

After studying Eisworth's works \cite{eisworth2},\cite{eisworth1}, we realized that his arguments may
be pushed further to completely clarify the situation at the successor of singular cardinals,
and indeed,  the main result of this paper reads as follows.

\begin{THM}\label{t11} For every singular cardinal $\lambda$, and every $\theta\le\lambda^+$, the following are equivalent:
\begin{enumerate}
\item $\lambda^+\not\rightarrow[\lambda^+]^2_\theta$;
\item $\lambda^+\not\rightarrow[\lambda^+;\lambda^+]^2_\theta$;
\item $\pr_1(\lambda^+,\lambda^+,\theta,\cf(\lambda))$.
\end{enumerate}
\end{THM}

We mention that the above yields an alternative proof to Theorem 5 of \cite{eisworth2} (that is, a proof that does not appeal to \cite{sh572}),
as well as an affirmative answer to the question
of Eisworth raised in \cite{eisworth3} of
whether the failure of $\pr_1(\lambda^+,\lambda^+,\lambda^+,\cf(\lambda))$
implies that any collection of less than $\cf(\lambda)$ many stationary subsets of $\lambda^+$ must reflect simultaneously.

The actual proof of Theorem \ref{t11} goes through establishing the existence of
a function of its own interest: a function that transforms rectangles $A\circledast B$ into squares $[C]^2$.\footnote{Here,
$A\circledast B$ stands for the set $\{(\alpha,\beta)\in A\times B\mid \alpha<\beta\}$, and $[C]^2:=C\circledast C$.}

\begin{THM}\label{t12} For every singular cardinal $\lambda$,
there exists a function $\rts:[\lambda^+]^2\rightarrow[\lambda^+]^2$ such that
for every cofinal subsets $A,B$ of $\lambda^+$, there exists a cofinal $C\s\lambda^+$ such that $\rts[A\circledast B]\supseteq [C]^2$.

Moreover:
\begin{itemize}
\item for every cofinal subsets $A,B$ of $\lambda^+$, and every $\kappa<\lambda$,
there exists a stationary subset $S\s E^{\lambda^+}_{>\kappa}$ such that $\rts[A\circledast B]\supseteq [S]^2$;\footnote{Here, $E^{\lambda^+}_{>\kappa}$ stands for the set $\{\alpha<\lambda^+\mid \cf(\alpha)>\kappa\}$.}
\item if $\alpha<\beta<\lambda^+$ and $\rts(\alpha,\beta)=(\alpha^*,\beta^*)$, then $\alpha^*\le\alpha<\beta^*\le\beta$.
\end{itemize}
\end{THM}

\subsection*{Acknowledgement}
We would like to express our intellectual debt to Eisworth, Shelah, and Todorcevic,
whom technology is in the heart of this paper.
While the results of this paper yields an answer to a question of Eisworth from \cite{eisworth3},
the proof of  $(2)\Rightarrow (3)$ in Theorem \ref{t11} builds on the main result from Eisworth's \cite{eisworth2}
and on Theorem \ref{t12}, whereas, the proof of Theorem \ref{t12} is just the outcome of several complications of Eisworth's arguments from \cite{eisworth1}.

\section{Transforming rectangles into squares}
In this section we prove Theorem \ref{t12}, where the proof splits into two cases --- the case $\cf(\lambda)>\omega$,
and the case $\cf(\lambda)=\omega$. In either case, we first define the function $\rts$, and then verify that it works.

We mention that the definition(s) of $\rts$ are virtually the same as that of the function $D$ from \cite[Theorem 1]{eisworth1},
and that our analysis of these functions is just a slight extension of Eisworth's analysis from \cite{eisworth1}.

\subsection{The uncountable cofinality case}
In this subsection, we prove Theorem \ref{t12}, for the case that $\lambda$ is a singular cardinal of uncountable cofinality.
We shall rely on Shelah's theorem on club guessing.
\begin{thm}[Shelah,  \cite{sh:g}]\label{31} For every cardinal $\lambda>\cf(\lambda)>\omega$,
there exists a sequence $\overrightarrow e=\langle e_\delta\mid \delta\in
E^{\lambda^+}_{\cf(\lambda)}\rangle$ such that:
\begin{enumerate}
\item $e_\delta$ is a club in $\delta$ of order-type $\cf(\lambda)$, for all $\delta\in E^{\lambda^+}_{\cf(\lambda)}$;
\item for every club $E\s\lambda^+$, there exist stationarily many $\delta\in E^{\lambda^+}_{\cf(\lambda)}$ for which
$\sup\{ \cf(\alpha)\mid \alpha\in e_\delta\cap E\}=\lambda$.
\end{enumerate}
\end{thm}

\begin{lemma}[Shelah, \cite{sh:g}]\label{32} Let $\lambda$ and $\overrightarrow e$ be as in the preceding theorem.
Then there exists a sequence $\overrightarrow C=\langle C_\alpha\mid
\alpha<\lambda^+\rangle$ such that for every $\alpha<\lambda^+$:
\begin{itemize}
\item $\max(C_{\alpha+1})=\alpha$;
\item $C_\alpha$ is a club in $\alpha$ of size $<\lambda$;
\item if $\delta\in C_\alpha\cap E^{\lambda^+}_{\cf(\lambda)}$, then $e_\delta\s C_\alpha$.
\end{itemize}
\end{lemma}

From now on, we fix a cardinal $\lambda$ and a sequence
$\overrightarrow e,\overrightarrow C$ as in the preceding, and shall be conducting \emph{minimal walks} \cite{MR2355670}
along this $\overrightarrow C$-sequence.
For $\alpha<\beta<\lambda^+$, let $\beta_0:=\beta$, and $\beta_{n+1}:=\min(C_{\beta_n}\bks\alpha)$
whenever $n<\omega$ and $\beta_n>\alpha$.
The decreasing sequence $\beta=\beta_0>\beta_1>\ldots\beta_{k}=\alpha$ is called the
$\overrightarrow C$-walk from $\beta$ down to $\alpha$, and we denote
\begin{itemize}
\item $\rho(\alpha,\beta):=k$;
\item $\Tr(\alpha,\beta):=\{ \beta_0,\ldots,\beta_k\}$;
\item $\tr(\alpha,\beta):=\Tr(\alpha,\beta)\bks\{\alpha\}$.
\end{itemize}

Next, by Shelah \cite{sh:g}, we may fix a scale  $(\overrightarrow\lambda,\overrightarrow f)$  for $\lambda$.
That is, a sequence $\overrightarrow\lambda=\langle \lambda_i\mid i<\cf(\lambda)\rangle$
of regular cardinals which is strictly increasing and cofinal in $\lambda$,
and a sequence $\overrightarrow f=\langle f_\alpha\mid\alpha<\lambda^+\rangle$
of elements from $\prod\overrightarrow\lambda$ which is strictly increasing and cofinal in $\left\langle \prod\overrightarrow\lambda,\le^*\right\rangle$.
So, $\alpha<\beta<\lambda^+$ implies that $\{ i<\cf(\lambda)\mid f_\alpha(i)\ge f_\beta(i)\}$ is of size $<\cf(\lambda)$.
In particular, it makes sense to consider the next ordinal:
$$\Delta(\alpha,\beta):=\sup\{ i<\cf(\lambda)\mid f_\alpha(i)\ge f_\beta(i)\}.$$

Let $\psi:\cf(\lambda)\rightarrow\omega\times\omega$  be some
surjection such that the pre-image of every element is cofinal in $\cf(\lambda)$.

Fix $\alpha<\beta<\lambda^+$, and let us define $\rts(\alpha,\beta)$.
Let $k_1$, $k_2$  be such that
$\psi(\Delta(\alpha,\beta))=(k_1,k_2).$
Next, let:
\begin{itemize}
\item $\beta_0>\dots>\beta_{n}$ denote the $\overrightarrow C$-walk from $\beta$ down to $\alpha$;
\item $\eta:=\max\{\sup(C_{\tau}\cap\beta_{k_1})\mid \tau\in\tr(\beta_{k_1},\beta)\}$;
\item $\alpha_0>\ldots>\alpha_{m}$ denote the $\overrightarrow C$-walk from $\alpha$ down to $\eta+1$;
\end{itemize}

We would like to assign $\rts(\alpha,\beta):=(\alpha_{k_2},\beta_{k_1})$.
Of course, if $k_1> n$, $\alpha<\eta+1$, or $k_2> m$,
then the above definition would not make sense. In such cases, we just put $\rts(\alpha,\beta):=(\alpha,\beta)$.

In Theorem \ref{t25} below, we prove that $\rts$ has the required transfer property.
The proof builds on a technical lemma (Lemma \ref{344}), whose statement requires the following piece of notation.

\begin{defn}\label{dddd} For  $A\s\lambda^+$, $\eta\le\epsilon<\lambda^+$, and $k<\omega$, let
$$A_{k,\eta}(\epsilon):=\left\{\alpha\in A\mid \rho(\epsilon,\alpha)=k, \max\{\sup(C_\tau\cap\epsilon)\mid \tau\in\tr(\epsilon,\alpha)\}=\eta\right\}.$$
\end{defn}
\begin{lemma}\label{344} If $A,B$ are cofinal subsets of $\lambda^+$, and $\theta<\lambda$,
then the following set is stationary:
$$S^{A,B}_{>\theta}:=\{\epsilon\in E^{\lambda^+}_{>\theta}\mid \exists(k_1,k_2,\eta)\in\omega\times\omega\times\epsilon\left( |B_{k_1,\eta}(\epsilon)|=|A_{k_2,\eta}(\epsilon)|=\lambda^+\right)\}.$$
\end{lemma}
\begin{proof} Suppose that $A,B$ are cofinal subsets of $\lambda^+$, $\theta<\lambda$,
and that $E$ is a club subset of $\lambda^+$.
We shall prove that $S^{A,B}_{>\theta}\cap E\not=\emptyset$.

Let $S$ denote the set of all $\delta\in E^{\lambda^+}_{\cf(\lambda)}$ for which there exists
an elementary submodel $M\prec\mathcal H(\chi)$ such that:\footnote{As usual, $\chi$ is some large enough regular cardinal which we fix throughout the whole paper,
$\mathcal H(\chi):=\langle H(\chi),\in,<_\chi\rangle$ is an hereditary fragment of the universe that admits a well-ordering $<_\chi$.}
\begin{itemize}
\item $\{\overrightarrow C,A,B\}\s M$;
\item $M\cap\lambda^+=\delta>\lambda$;
\item $\sup\{ \cf(\epsilon)\mid \epsilon\in e_\delta\cap E\}=\lambda$.
\end{itemize}

By the choice of the club guessing sequence $\overrightarrow e$, the set $S$ is stationary.
Put $D:=\{\varepsilon<\lambda^+\mid \sup(S\cap\varepsilon)=\varepsilon\}$.
Then $D$ is a club,
so let us fix some $\nu\in E^{\lambda^+}_{\cf(\lambda)}$ such that
$\sup\{ \cf(\varepsilon)\mid \varepsilon\in e_\nu\cap D\}=\lambda$.

Put $\beta:=\min(B\bks\nu+1)$, and $\alpha:=\min(A\bks\nu+1)$.
Let
\begin{itemize}
\item $\beta=\beta_0>\ldots>\beta_{n+1}=\nu$ denote the $\overrightarrow C$-walk from $\beta$ down to $\nu$;
\item $\alpha=\alpha_0>\ldots>\alpha_{m+1}=\nu$ denote the $\overrightarrow C$-walk from $\alpha$ down to $\nu$;
\item $\gamma:=\max(\{\sup(C_{\beta_i}\cap\nu)\mid i<n\}\cup \{\sup(C_{\alpha_j}\cap\nu)\mid j<m\}).$
\end{itemize}

Since $\nu\not\in C_{\beta_i}\cup C_{\alpha_j}$ for all $i<n$ and  $j<m$,
we get that $\gamma<\nu$, so let us pick some large enough
$\varepsilon\in e_\nu\cap D$ above $\gamma$ such that
$\cf(\varepsilon)>\max\{|C_{\beta_n}|,|C_{\alpha_m}|\}$.

\begin{claim}\label{c241} $\tr(\varepsilon,\beta)=\tr(\nu,\beta)$, and $\tr(\varepsilon,\alpha)=\tr(\nu,\alpha)$.
\end{claim}
\begin{proof}
As  $(\gamma,\nu)\cap C_{\beta_i}=\emptyset$
for all $i<n$, and $\varepsilon\in(\gamma,\nu)$, we get that
$\min(C_{\beta_i}\bks\varepsilon)=\min(C_{\beta_i}\bks\nu)$ for all $i<n$,
and hence $\Tr(\varepsilon,\beta)\bks\beta_n=\{\beta_i\mid i\le n\}=\tr(\nu,\beta)$.
Likewise, by $(\gamma,\nu)\cap C_{\alpha_j}=\emptyset$
for all $j<m$, we get that
$\Tr(\varepsilon,\alpha)\bks\alpha_m=\tr(\nu,\alpha)$.

Since $\nu\in
C_{\beta_n}\cap C_{\alpha_m}\cap E^{\lambda^+}_{\cf(\lambda)}$, we have $e_\nu\s
C_{\beta_n}\cap C_{\alpha_m}$, and in particular $\varepsilon\in C_{\beta_n}\cap C_{\alpha_m}$. So
$\rho(\varepsilon,\beta)=n+1$, $\rho(\varepsilon,\alpha)=m+1$,
and hence
$\tr(\varepsilon,\beta)=\tr(\nu,\beta)$, $\tr(\varepsilon,\alpha)=\tr(\nu,\alpha)$.
\end{proof}

Put $$\gamma':=\max\{\sup(C_{\tau}\cap\varepsilon)\mid
\tau\in\tr(\varepsilon,\beta)\cup\tr(\varepsilon,\alpha)\}.$$

Since $\gamma<\varepsilon<\nu$, we have $\gamma'=\max\{\gamma,\sup(C_{\beta_n}\cap\varepsilon),\sup(C_{\alpha_m}\cap\varepsilon)\}$.
Since $\varepsilon\in C_{\beta_n}\cap C_{\alpha_m}$ and $\cf(\varepsilon)>\max\{|C_{\beta_n}|,|C_{\alpha_m}|\}$, we
get that $\max\{\sup(C_{\beta_n}\cap\varepsilon),\sup(C_{\alpha_m}\cap\varepsilon)\}<\varepsilon$,
and hence  $\gamma'<\varepsilon$.
Since $\varepsilon\in e_\nu\cap D$, we now pick some large enough $\delta\in S$ satisfying $\gamma'<\delta<\varepsilon$,
and consider additional walks. Let:
\begin{itemize}
\item $\beta=\beta'_0>\ldots>\beta'_{n'+1}=\delta$ denote the $\overrightarrow C$-walk
from $\beta$ down to $\delta$;
\item $\alpha=\alpha'_0>\ldots>\alpha'_{m'+1}=\delta$ denote the $\overrightarrow C$-walk from $\alpha$ down to $\delta$.
\end{itemize}

\begin{claim} \begin{enumerate}
\item $\tr(\delta,\beta)=\tr(\delta,\varepsilon)\cup\tr(\varepsilon,\beta)$;
\item $\tr(\delta,\alpha)=\tr(\delta,\varepsilon)\cup\tr(\varepsilon,\alpha)$.
\end{enumerate}
\end{claim}
\begin{proof}
By definition of $\gamma'$ and since $\delta\in(\gamma',\varepsilon)$, we get that $\beta'_i=\beta_i$ for all $i\le n$,
and hence $\beta'_{n+1}=\varepsilon$. So $\tr(\varepsilon,\beta)=\{\beta'_0,\ldots\beta'_n\}$, and $\tr(\delta,\varepsilon)=\{\beta'_{n+1},\ldots,\beta'_{n'}\}$.

Likewise, $\alpha'_j=\alpha_j$ for all $j\le m$, and hence $\alpha'_{m+1}=\varepsilon$.
So $\tr(\varepsilon,\alpha)=\{\alpha'_0,\ldots,\alpha'_m\}$,
and $\tr(\delta,\varepsilon)=\{\alpha'_{m+1},\ldots,\alpha'_{m'}\}$.
\end{proof}
Next, denote
\begin{itemize}
\item  $\zeta:=\min(\tr(\delta,\varepsilon))$;
\item $\gamma_\alpha:=\max\{\sup(C_{\alpha'_j}\cap\delta)\mid j<m'\}$;
\item $\gamma_\beta:=\max\{\sup(C_{\beta'_i}\cap\delta)\mid i<n'\}$.
\end{itemize}
Evidently, $\gamma_\alpha<\delta$, $\gamma_\beta<\delta$. Note that by the preceding claim, we have $\beta'_{n'}=\alpha'_{m'}=\zeta$.

Pick a large enough $\epsilon\in e_\delta\cap E$ such that
$\sup(e_\delta\cap\epsilon)>\max\{\gamma_\alpha,\gamma_\beta\}$, and
$\cf(\epsilon)>\max\{|C_{\zeta}|,\theta\}$.

\begin{claim} $\tr(\epsilon,\beta)=\tr(\delta,\beta)$, and $\tr(\epsilon,\alpha)=\tr(\delta,\alpha)$.
\end{claim}
\begin{proof} As in the proof of Claim \ref{c241}.
\end{proof}

Denote:
\begin{itemize}
\item $\eta_\alpha:=\max\{\sup(C_\tau\cap\epsilon)\mid \tau\in\tr(\epsilon,\alpha)\}$;
\item $\eta_\beta:=\max\{\sup(C_\tau\cap\epsilon)\mid \tau\in\tr(\epsilon,\beta)\}$.
\end{itemize}

Since $\min(\tr(\epsilon,\beta))=\min(\tr(\delta,\beta))=\min(\tr(\delta,\varepsilon))=\zeta$,
and $\gamma_\beta<\epsilon<\delta$, we get that $\eta_\beta=\max\{\gamma_\beta,\sup(C_\zeta\cap\epsilon)\}$.
Since $\delta\in C_\zeta\cap E^{\lambda^+}_{\cf(\lambda)}$, we get that $\sup(C_\zeta\cap\epsilon)\ge\sup(e_\delta\cap\epsilon)>\gamma_\beta$,
and hence $\eta_\beta=\sup(C_\zeta\cap\epsilon)$.

By a similar consideration, we get that $\eta_\alpha=\sup(C_\zeta\cap\epsilon)$.

Denote $\eta:=\sup(C_\zeta\cap\epsilon)$.
Then, we have just established that
\begin{itemize}
\item $\beta\in B_{k_1,\eta}(\epsilon)$, for $k_1:=\rho(\epsilon,\beta)$, and
\item $\alpha\in A_{k_2,\eta}(\epsilon)$, for $k_2:=\rho(\epsilon,\alpha)$.
\end{itemize}

By the choice of $\delta$, let us now pick an elementary submodel $M\prec\mathcal H(\chi)$
such that $\{A,B,\overrightarrow C\}\s M$ and $M\cap\lambda^+=\delta$.

By $\eta\le\epsilon<\delta$,
the sets $B_{k_1,\eta}(\epsilon)$ and $A_{k_2,\eta}(\epsilon)$ are definable from parameters in $M$,
and hence belongs to $M$. Since $|M|=\lambda$ and $\beta\in B_{k_1,\eta}(\epsilon)\bks M$,
we get that $|B_{k_1,\eta}(\epsilon)|=\lambda^+$.
Likewise, $|A_{k_2,\eta}(\epsilon)|=\lambda^+$.

By the choice of $\epsilon$,
we have $\cf(\epsilon)>|C_\zeta|$, and hence $\eta<\epsilon$.
Altogether,  $\epsilon\in S^{A,B}_{>\theta}$.
Recalling that $\epsilon\in e_\delta\cap E$, we conclude that
$S^{A,B}_{>\theta}\cap E\not=\emptyset$.
\end{proof}

\begin{thm}\label{t25} For every cofinal subsets $A,B$ of $\lambda^+$, and every $\theta<\lambda$,
there exists a stationary subset $S^*\s E^{\lambda^+}_{>\theta}$ such that $\rts[A\circledast B]\supseteq [S^*]^2$.
\end{thm}
\begin{proof}
Suppose that $A,B$ are cofinal subsets of $\lambda^+$, and that $\theta<\lambda$.
By Lemma \ref{344}, $S^{A,B}_{>\theta}$ is stationary, so we appeal to Fodor's
lemma and fix some $k_1,k_2<\omega$ and $\eta<\lambda^+$ such that the
following set is stationary:
$$S:=\{\epsilon\in E^{\lambda^+}_{>\theta}\mid |B_{k_1,\eta}(\epsilon)|=|A_{k_2,\eta}(\epsilon)|=\lambda^+\}.$$

Next, let $\langle M_i\mid i<\lambda^+\rangle$ be a continuous
$\in$-chain of elementary submodels of $\mathcal H(\chi)$ such that:
\begin{enumerate}
\item $\{\overrightarrow \lambda,\overrightarrow f,\overrightarrow C,A,\eta\}\s  M_0$;
\item $|M_i|=\lambda$ for all $i<\lambda$;
\end{enumerate}

Let $S^*$ be the following stationary subset of $E^{\lambda^+}_{>\theta}$:
$$S^*:=\{ \epsilon\in S\mid M_\epsilon\cap\lambda^+=\epsilon\}.$$

\begin{claim}\label{c251} For every $\alpha^*<\beta^*$ from $S^*$,
there exists $\alpha\in A$, $\beta\in B$ such that all of the followings holds:
\begin{enumerate}
\item $\eta<\alpha^*<\alpha<\beta^*<\beta$;
\item $\rho(\beta^*,\beta)=k_1$;
\item $\rho(\alpha^*,\alpha)=k_2$;
\item $\psi(\Delta(\alpha,\beta))=(k_1,k_2)$;
\item $\max\{\sup(C_\tau\cap\beta^*)\mid \tau\in \tr(\beta^*,\beta)\}=\eta$;
\item $\max\{\sup(C_\tau\cap\alpha^*)\mid \tau\in \tr(\alpha^*,\alpha)\}=\eta$;
\end{enumerate}
\end{claim}
\begin{proof}
As $\beta^*\in S$, let us pick some $\beta\in B_{k_1,\eta}(\beta^*)$ with $\beta^*<\beta$.
Next, let $M=\sk_{\mathcal H(\chi)}(x)$ denote the Skolem hull in $\mathcal H(\chi)$ of the set $x:=\{(\overrightarrow\lambda,\overrightarrow f),\overrightarrow
C,A,\eta,\alpha^*\}\cup\cf(\lambda)+1$. As $|M|<\lambda=\sup(\overrightarrow\lambda)$,
it is now reasonable to consider the characteristic function of $M$ on $\overrightarrow\lambda$.
We remind the reader that the latter, which we denote by $\ch^{\overrightarrow\lambda}_M$, stands for the unique
 element of $\prod\overrightarrow\lambda$ that satisfies for all $i<\cf(\lambda)$:
$$\ch^{\overrightarrow\lambda}_M(i)=\begin{cases}\sup(M\cap\lambda_i),&\sup(M\cap\lambda_i)<\lambda_i\\
0,&\text{otherwise}\end{cases}.$$

Next, note that $M\in M_{\beta^*}$.
Indeed, as $x\s M_{\alpha^*+1}\prec\mathcal H(\chi)$, we get that
$\sk_{\mathcal H(\chi)}(x)=\sk_{M_{\alpha^*+1}}(x)$. So, $M$ is definable from $x$ and $M_{\alpha^*+1}$,
and hence $M\in M_{\alpha^*+2}\s M_{\beta^*}$.

By $M\in M_{\beta^*}$, we get that $\ch^{\overrightarrow\lambda}_M\in M_{\beta^*}$, and hence
$$M_{\beta^*}\models \exists\delta<\lambda^+(\ch^{\overrightarrow\lambda}_M<^*f_\delta).$$
As $M_{\beta^*}\cap\lambda^+=\beta^*$, we may find some
$\delta<\beta^*$ such that $\ch^{\overrightarrow\lambda}_M<^*f_\delta$.
By $\delta<\beta^*<\beta$, we infer that $\ch^{\overrightarrow\lambda}_M<^*f_\beta$.

Denote $A':=A_{k_2,\eta}(\alpha^*)\bks(\alpha^*+1)$.
For all $i<\cf(\lambda)$, put $\Gamma_i:=\{f_\alpha(i)\mid
\alpha\in A'\}$. As $A'$ is cofinal in $\lambda^+$, $\langle
f_\alpha\mid \alpha\in A'\rangle$ is cofinal in
$\langle\prod\overrightarrow\lambda,<^*\rangle$, and hence
$\sup(\Gamma_i)=\lambda_i$ for co-boundedly many $i<\cf(\lambda)$.
Thus, let us pick some large enough ${i'}<\cf(\lambda)$ such that:
\begin{itemize}
\item $\ch^{\overrightarrow\lambda}_M(i)<f_\beta(i)$ whenever $i'<i<\cf(\lambda)$;
\item $\sup(\Gamma_{{i'}})=\lambda_{{i'}}>|M|$.
\item $\psi(i')=(k_1,k_2)$.
\end{itemize}

Put $\varepsilon:=f_\beta(i')$. As $\varepsilon\in\lambda_{i'}$,
we define $N:=\sk_{\mathcal H(\chi)}(M\cup\lambda_{i'})$.
By $x\s M\s N$, we have
$\{\overrightarrow C,A,k_2,\eta,\alpha^*,\varepsilon\}\s N$.
So $A'\in N$, and
$$N\models \exists\alpha\in A'\left(f_\alpha(i')>\varepsilon\right).$$

 Thus, let us pick some $\alpha\in  A'\cap N$ such that
$f_\alpha(i')>\varepsilon$. Since $\alpha\in N\s M_{\beta^*}$, we
have $\alpha<\beta^*$.

So $\alpha\in A_{k_2,\eta}(\alpha^*)$, $\beta\in B_{k_1,\eta}(\beta^*)$,
$\eta<\alpha^*<\alpha<\beta^*<\beta$, and hence we are left with verifying that $\psi(\Delta(\alpha,\beta))=(k_1,k_2)$.

\begin{sclaim}\label{c251} Let $\ch_N^{\overrightarrow \lambda}$ denote the characteristic function of $N$ on $\overrightarrow\lambda$.
Then $\ch^{\overrightarrow\lambda}_N(i)=\ch^{\overrightarrow\lambda}_M(i)$ whenever $i'<i<\cf(\lambda)$.
\end{sclaim}
\begin{proof} This is an instance of Baumgartner's lemma \cite{baum},
stating that if $M\preceq\mathcal H(\chi)$ is some structure, and $\kappa<\mu$ are regular cardinals in $M$,
then $\sup(N\cap\mu)=\sup(M\cap\mu)$, for $N:=\sk_{\mathcal H(\chi)}(M\cup\kappa)$.
\end{proof}

Next, note that for all $i$ with $i'<i<\cf(\lambda)$, we have:
\begin{itemize}
\item $f_\alpha(i)\le\ch^{\overrightarrow\lambda}_N(i)$, since $\alpha\in N$;
\item $\ch^{\overrightarrow\lambda}_N(i)=\ch^{\overrightarrow\lambda}_M(i)$, by the preceding subclaim;
\item $\ch^{\overrightarrow\lambda}_M(i)<f_\beta(i)$, by $i>i'$ and the choice of $i'$.
\end{itemize}

So $f_\alpha(i)<f_\beta(i)$ whenever $i<i'<\cf(\lambda)$, and hence $\Delta(\alpha,\beta)\le i'$.
By $f_\alpha(i')>\varepsilon=f_\beta(i')$, we also get that $\Delta(\alpha,\beta)\ge i'$.

It follows that $\Delta(\alpha,\beta)= i'$, and hence $\psi(\Delta(\alpha,\beta))=\psi(i')=(k_1,k_2)$.
\end{proof}

\begin{claim} $[S^*]^2\s\rts``[A\otimes B]$.
\end{claim}
\begin{proof}
Suppose that $(\alpha^*,\beta^*)$ is a given element of
$[S^*]^2$. Fix $\alpha\in A$, $\beta\in B$
as in  Claim \ref{c251}, and let us show that
$\rts(\alpha,\beta)=(\alpha^*,\beta^*)$.

Let:
\begin{itemize}
\item $\beta_0>\ldots>\beta_{n}$ denote the
$\overrightarrow C$-walk from $\beta$ down to $\alpha$;
\item $\beta'_0>\ldots>\beta'_{k_1}$ denote the
$\overrightarrow C$-walk from $\beta$ down to $\beta^*$;
\item $\alpha_0>\ldots>\alpha_{m}$ denote the
$\overrightarrow C$-walk from $\alpha$ down to $\eta+1$;\footnote{Notice that by the choice of $\beta$,
we have $\eta<\alpha^*<\alpha$, and hence  $\alpha>\eta+1$.}
\item $\alpha'_0>\ldots>\alpha'_{k_2}$ denote the
$\overrightarrow C$-walk from $\alpha$ down to $\alpha^*$.
\end{itemize}

By $\beta_0'=\beta_0=\beta$ and $\sup(C_{\beta'_i}\cap\beta^*)\le \eta<\alpha<\beta^*$ for all $i<k_1$,
we get that $\beta_{i+1}'=\min(C_{\beta_{i'}}\bks \beta^*)=\min(C_{\beta_i}\bks\alpha)=\beta_{i+1}$ for all $i<k_1$.
In particular, $n\ge k_1$, $\beta_{k_1}=\beta^*$, and
$$\eta=\max\{\sup(C_{\tau}\cap\beta_{k_1})\mid\tau\in\tr(\beta_{k_1},\beta)\}.$$

By $\alpha_0'=\alpha_0$ and $\sup(C_{\alpha'_j}\cap\alpha^*)\le \eta<\eta+1\le\alpha^*$ for all $j<k_2$,
we get that $\alpha_{j+1}'=\min(C_{\alpha_{j'}}\bks\alpha^*)=\min(C_{\alpha_j}\bks\eta+1)=\alpha_{j+1}$ for all $j<k_2$.
In particular, $m\ge k_2$ and $\alpha_{k_2}=\alpha^*$.

It now follows from $\psi(\Delta(\alpha,\beta))=(k_1,k_2)$ and
the definition of $\rts$, that $\rts(\alpha,\beta)=(\alpha^*,\beta^*)$.
\end{proof}
\end{proof}

\subsection{The countable cofinality case}
In this subsection, we prove Theorem \ref{t12}, for the case that $\lambda$ is a singular cardinal of countable cofinality.
Here, instead of appealing to Theorem \ref{31}, we shall make use of Eiswroth's theorem concerning \emph{off-center club guessing}.

\begin{thm}[Eisworth, \cite{eisworth4}] Suppose that $\lambda>\cf(\lambda)=\omega$.

Then there exists an increasing sequence of successor cardinals
$\overrightarrow \mu=\langle \mu_m\mid m<\omega\rangle$, and a matrix
$\overrightarrow e=\langle e_\delta^m \mid \delta\in E^{\lambda^+}_{\omega_1}, m<\omega\rangle$ such that:
\begin{enumerate}
\item $\{ e^m_\delta\mid m<\omega\}$ is an increasing chain of club subsets of $\delta$;
\item $|e^m_\delta|\le\mu_m$ for all $\delta$ and $m$;
\item for every club $E\s\lambda^+$, there exist stationarily many $\delta\in E^{\lambda^+}_{\omega_1}$
such that  $\sup(e^m_\delta\cap E\cap E^{\delta}_{>\mu_m})=\delta$ for all $m<\omega$;
\item $\sup\overrightarrow\mu=\lambda$.
\end{enumerate}
\end{thm}

The next lemma is analogous to that of  Lemma \ref{32}.
\begin{lemma}[Eisworth, \cite{eisworth4}] Let  $\lambda$ and $\overrightarrow e,\overrightarrow\mu$ be as in the preceding theorem.
Then there exists a matrix $\overrightarrow C=\langle C^m_\alpha\mid \alpha<\lambda^+, m<\omega\rangle$ such that for every
$\alpha<\lambda^+$:
\begin{itemize}
\item $\max(C^m_{\alpha+1})=\alpha$ for all $m<\omega$;
\item $\{C^m_\alpha\mid m<\omega\}$ is a chain of club subsets of $\alpha$;
\item $|C^m_\alpha|\le\max\{\mu_m,\cf(\alpha)\}$;
\item if $\delta\in C^m_\alpha\cap E^{\lambda^+}_{\omega_1}$, then $e^m_\delta\s C^m_\alpha$.
\end{itemize}
\end{lemma}

From now on, we fix a cardinal $\lambda$ and sequences
$\overrightarrow e,\overrightarrow C, \overrightarrow \mu$ as above.

\begin{notation} if $\sigma\in{}^{<\omega}\omega$ is a non-empty sequence and $m<\omega$,
we define $\ell(\sigma):=\dom(\sigma)$, $\sigma^\frown  m$ as $\sigma\cup\{(\ell(\sigma),m)\}$,
and $\lift(\sigma)$ as the unique element of ${}^\omega\omega$ satisfying for all $n<\omega$,
$$\lift(\sigma)(n)=\begin{cases}\sigma(n),&n\in\dom(\sigma)\\
\sigma(\max(\dom(\sigma)),&\text{otherwise}\end{cases}.$$
\end{notation}

In this subsection, we shall be interested in conducting \emph{generalized walks}:
for $\alpha<\beta<\lambda^+$ and $\sigma\in{}^{<\omega}\omega$, let $\beta_0:=\beta$, and $\beta_{n+1}:=\min(C^{\lift(\sigma)(n)}_{\beta_n}\bks\alpha)$
whenever $n<\omega$ and $\beta_n>\alpha$.
The outcome $\beta=\beta_0>\ldots>\beta_{k+1}=\alpha$ is considered as \emph{the
$(\overrightarrow C,\sigma)$-walk from $\beta$ down to $\alpha$}, and  we denote:
\begin{itemize}
\item $\rho_\sigma(\alpha,\beta):=k+1$;
\item $\trs(\alpha,\beta):=\{ (\lift(\sigma)(i),\beta_i)\mid i\le k\}$;
\item $\Tr_\sigma(\alpha,\beta):=\{ (\lift(\sigma)(i),\beta_i)\mid i\le k+1\}$.
\end{itemize}

Fix a surjection $\psi:\omega\rightarrow{}^{<\omega}\omega\times{}^{<\omega}\omega$ such that the pre-image of every element is infinite.
Let $(\overrightarrow\lambda,\overrightarrow f)$ be
a scale for $\lambda$, and let, as in the previous subsection, for all $\alpha<\beta<\lambda^+$:
$$\Delta(\alpha,\beta):=\sup\{ i<\cf(\lambda)\mid f_\alpha(i)\ge f_\beta(i)\}.$$

For $\alpha<\beta<\lambda^+$, we now define $\rts(\alpha,\beta)$. Let $\sigma_1,\sigma_2$ be such that
$\psi(\Delta(\alpha,\beta))=(\sigma_1,\sigma_2)$.
Next, let:
\begin{itemize}
\item $\beta_0>\dots>\beta_{n}$ denote the
$(\overrightarrow{C},\sigma_1)$-walk from $\beta$ down to $\alpha$;
\item $\eta:=\max\{\sup(C^i_{\tau}\cap\beta_{\ell(\sigma_1)})\mid (i,\tau)\in\Tr^\circ_{\sigma_1}(\beta_{\ell(\sigma_1)},\beta)\}$;
\item  $\alpha_0>\ldots>\alpha_{m}$ denote the
$(\overrightarrow{C},\sigma_2)$-walk from $\alpha$ down to $\eta+1$;
\end{itemize}

We would like to assign $\rts(\alpha,\beta):=(\alpha_{\ell(\sigma_2)},\beta_{\ell(\sigma_1)})$.
Of course, if $\ell(\sigma_1)> n$, $\alpha<\eta+1$, or $\ell(\sigma_2)> m$,
then the above definition would not make sense. In such cases, we just put $\rts(\alpha,\beta):=(\alpha,\beta)$.

The next definition is parallel to Definition \ref{dddd}.

\begin{defn} For  $A\s\lambda^+$, $\eta\le\epsilon<\lambda^+$, and $\sigma\in{}^{<\omega}\omega$, let
$$A_{\sigma,\eta}(\epsilon):=\left\{\alpha\in A\mid \rho_\sigma(\epsilon,\alpha)=\ell(\sigma), \max\{\sup(C^i_\tau\cap\epsilon)\mid (i,\tau)\in\trs(\epsilon,\alpha)\}=\eta\right\}.$$
\end{defn}

\begin{lemma}\label{34} If $A,B$ are cofinal subsets of $\lambda^+$, and $\theta<\lambda$,
then the following set is stationary:
$$S^{A,B}_{>\theta}:=\{\epsilon\in E^{\lambda^+}_{>\theta}\mid
\exists(\sigma_1,\sigma_2,\eta)\in{}^{<\omega}\omega{}\times^{<\omega}\omega\times\epsilon\left( |B_{\sigma_1,\eta}(\epsilon)|=|A_{\sigma_2,\eta}(\epsilon)|=\lambda^+\right)\}.$$
\end{lemma}
\begin{proof}  Suppose that $A,B$ are cofinal subsets of $\lambda^+$, and $\theta<\lambda$.
 Let $E$ be an arbitrary club subset of $\lambda^+$,
and let us prove that $S^{A,B}_{>\theta}\cap E\not=\emptyset$.

Let $S$ denote the set of all $\delta\in E^{\lambda^+}_{\omega_1}$ for which there exists
an elementary submodel $M\prec\mathcal H(\chi)$
such that:
\begin{itemize}
\item $\{\overrightarrow C,A,B\}\s M$;
\item $M\cap\lambda^+=\delta>\lambda$;
\item $\sup(e^m_\delta\cap E\cap E^{\delta}_{>\mu_m})=\delta$ for all $m<\omega$.
\end{itemize}

Then $S$ is a stationary set, and  $D:=\{\varepsilon<\lambda^+\mid \sup(S\cap\varepsilon)=\varepsilon\}$ is a club.
By the choice of the club guessing sequence $\overrightarrow e$,
let us fix some $\nu\in E^{\lambda^+}_{\omega_1}$ such that $\sup(e^m_\nu\cap E\cap E^\nu_{>\mu_m})=\nu$ for all $m<\omega$.

Let $\beta:=\min(B\bks\nu+1)$, and $\alpha:=\min(A\bks\nu+1)$.
For all $m<\omega$, let
\begin{itemize}
\item $\beta=\beta^m_0>\ldots>\beta^m_{n(m)+1}=\nu$ denote the $(\overrightarrow{C},\langle m\rangle)$-walk
from $\beta$ down to $\nu$;
\item $\alpha=\alpha^m_0\ldots>\alpha^m_{k(m)+1}=\nu$ denote the $(\overrightarrow{C},\langle m\rangle)$-walk
from $\alpha$ down to $\nu$.
\end{itemize}

Clearly, the sequence $\langle \beta_0^m\mid m<\omega\rangle$ is a
constant sequence. As $\{C^m_\beta\mid m<\omega\}$ is an increasing chain,
the sequence $\langle \min(C^m_\beta\bks\nu)\mid m<\omega\rangle$
stabilizes, and hence the sequence $\langle \beta^m_1\mid
m<\omega\rangle$ is eventually constant.
Of course, an iterative application of this last observation yields that the sequence
$\langle \rng(\trm(\nu,\beta)) \mid m<\omega\rangle$ is eventually constant.
By the same kind of consideration, we get that the sequence
$\langle \rng(\trm(\nu,\alpha)) \mid m<\omega\rangle$ is eventually constant.
Thus, let us pick some large enough $m^*$ such that $\langle (\rng(\trm(\nu,\beta)),\rng(\trm(\nu,\alpha)))\mid m^*\le m<\omega\rangle$ is a
constant sequence.

Next, pick an even larger $m<\omega$ so that $m>m^*$, and
$$\mu_m\ge \max\{\cf(\beta^{m^*}_{n(m^*)}),\cf(\alpha^{m^*}_{k(m^*)})\}.$$

For notational simplicity, denote $n:=n(m), k:=k(m), \beta_i:=\beta^m_i, \alpha_j:=\alpha^m_j$.
Put
$$\gamma:=\max(\{\sup(C^m_{\beta_i}\cap\nu)\mid i<n\}\cup \{\sup(C^m_{\alpha_j}\cap\nu)\mid j<k\}).$$

Since $\gamma<\nu$, let us  pick a large enough $\varepsilon\in e^m_\nu\cap D$
such that $\varepsilon>\gamma$ and $\cf(\varepsilon)>\mu_m$.

\begin{claim}\label{291} \begin{enumerate}
\item $\trm(\varepsilon,\beta)=\{(m,\beta_i)\mid i\le n\}$;
\item $\trm(\varepsilon,\alpha)=\{(m,\alpha_j)\mid j\le k\}$.
\end{enumerate}
\end{claim}
\begin{proof}
As  $(\gamma,\nu)\cap C^m_{\beta_i}=\emptyset$ for all $i<n$, and
$\varepsilon\in(\gamma,\nu)$, we get that
$\min(C^m_{\beta_i}\bks\varepsilon)=\min(C^m_{\beta_i}\bks\nu)$ for all
$i<n$, and hence $\rng(\Tr_{\langle m\rangle}(\varepsilon,\beta))\bks\beta_n=\{\beta_i\mid i\le n\}$. Likewise, by $\varepsilon\in(\gamma,\nu)$
and $(\gamma,\nu)\cap C^m_{\alpha_j}=\emptyset$ for all $j<k$, we
get that $\rng(\Tr_{\langle m\rangle}(\varepsilon,\alpha))\bks\alpha_k=\{\alpha_j\mid j\le k\}$.

Since $\nu\in C^m_{\beta_n}\cap C^m_{\alpha_k}\cap E^{\lambda^+}_{\omega_1}$, we
have $e^m_\nu\s C^m_{\beta_n}\cap C^m_{\alpha_k}$, which implies
$\varepsilon\in C^m_{\beta_n}\cap C^m_{\alpha_k}$.
So $\rho_{\langle m\rangle}(\varepsilon,\beta)=n+1$,
$\rho_{\langle m\rangle}(\varepsilon,\alpha)=k+1$, and the assertion of the claim follows immediately.
\end{proof}
Put $$\gamma':=\max\{\sup(C^i_{\tau}\cap\varepsilon)\mid
(i,\tau)\in\trm(\varepsilon,\beta)\cup\trm(\varepsilon,\alpha)\}.$$

Since $\gamma<\varepsilon<\nu$, we have $\gamma'=\max\{\gamma,\sup(C^m_{\beta_n}\cap\varepsilon),\sup(C^m_{\alpha_k}\cap\varepsilon)\}$.
Recalling the choice of $m$, and the fact that $|C^i_\tau|\le\max\{\mu_i,\cf(\tau)\}$ for
all $i$ and $\tau$, we infer that:
$$\max\{|C^{m}_{\beta_{n}}|,|C^{m}_{\alpha_{k}}|\}\le\max\{\mu_{m},\cf(\beta^{m^*}_{n(m^*)}),\cf(\alpha^{m^*}_{k(m^*)})\}=\mu_m.$$

By the choice of $\varepsilon$, we have $\mu_m<\cf(\varepsilon)$,
and hence $\sup(C^m_{\beta_n}\cap\varepsilon)<\varepsilon$,
$\sup(C^m_{\alpha_k}\cap\varepsilon)<\varepsilon$. Altogether, $\gamma'<\varepsilon$.

Consider the constant functions $\sigma_1:n+1\rightarrow\{m\}, \sigma_2:k+1\rightarrow\{m\}$,
and note that $\ell(\sigma_1)=\rho_{\langle m\rangle}(\varepsilon,\beta)$, $\ell(\sigma_2)=\rho_{\langle m\rangle}(\varepsilon,\alpha)$.

Since $\varepsilon\in e^m_\nu\cap D$, we now pick a large enough $\delta\in S$ satisfying $\gamma'<\delta<\varepsilon$,
and consider additional walks. For all $w<\omega$, let:
\begin{itemize}
\item $\beta=\beta^{\cdot w}_0>\ldots>\beta^{\cdot w}_{n(w)+1}=\delta$ denote the $(\overrightarrow{C},\sigma_1{}^\frown w)$-walk
from $\beta$ down to $\delta$;
\item $\alpha=\alpha^{\cdot w}_0>\ldots\alpha^{\cdot w}_{k(w)+1}=\delta$ denote the $(\overrightarrow{C},\sigma_2{}^\frown w)$-walk
from $\alpha$ down to $\delta$.
\end{itemize}

\begin{claim}\label{292} For every $w<\omega$:\begin{enumerate}
\item $\Tr^\circ_{\sigma_1{}^\frown w}(\delta,\beta)=\trw(\delta,\varepsilon)\cup\trm(\varepsilon,\beta)$;
\item $\Tr^\circ_{\sigma_2{}^\frown w}(\delta,\alpha)=\trw(\delta,\varepsilon)\cup\trm(\varepsilon,\alpha)$.
\end{enumerate}
\end{claim}
\begin{proof} Fix $w<\omega$.
By definition of $\gamma'$ and since $\delta\in(\gamma',\varepsilon)$, we get that $\beta^{\cdot w}_i=\beta_i$ for all $i\le n$,
and hence $\beta^{\cdot w}_{n+1}=\varepsilon$. So $\rng(\trm(\varepsilon,\beta))=\{\beta^{\cdot w}_0,\ldots\beta^{\cdot w}_n\}$,
and $\rng(\trw(\delta,\varepsilon))=\{\beta^{\cdot w}_{n+1},\ldots,\beta^{\cdot w}_{n(w)}\}$.

Likewise, $\alpha^{\cdot w}_j=\alpha_j$ for all $j\le k$, and hence $\alpha^w_{k+1}=\varepsilon$.
So $\rng(\trm(\varepsilon,\alpha))=\{\alpha^{\cdot w}_0,\ldots,\alpha^{\cdot w}_k\}$,
and $\rng(\trw(\delta,\varepsilon))=\{\alpha^{\cdot w}_{k+1},\ldots,\alpha^{\cdot w}_{k(w)}\}$.
\end{proof}

Fix a large enough $w^*<\omega$ for which
 $\langle \rng(\trw(\delta,\varepsilon)))\mid w^*\le w<\omega\rangle$ is a
constant sequence.
Also, let:
\begin{itemize}
\item  $\zeta:=\min(\rng(\Tr^\circ_{\langle w^*\rangle}(\delta,\varepsilon)))$;
\item $\gamma_\alpha:=\max\{\sup(C^i_{\tau}\cap\delta)\mid (i,\tau)\in \Tr^\circ_{\sigma_2{}^\frown w^*}(\delta,\alpha), \delta\not\in C^i_\tau\}$;
\item $\gamma_\beta:=\max\{\sup(C^i_{\tau}\cap\delta)\mid (i,\tau)\in \Tr^\circ_{\sigma_1{}^\frown w^*}(\delta,\beta), \delta\not\in C^i_\tau\}$.
\end{itemize}
Evidently, $\gamma_\alpha<\delta$, $\gamma_\beta<\delta$.
Now, pick a large enough $w<\omega$ so that $w>w^*$, and $\mu_w>\cf(\zeta)$.
Then, pick a large enough $\epsilon\in e^w_\delta\cap E$ such that
$\sup(e^w_\delta\cap\epsilon)>\max\{\gamma_\alpha,\gamma_\beta\}$, and
$\cf(\epsilon)>\max\{\mu_w,\theta\}$.

\begin{claim}\label{293} $\Tr^\circ_{\sigma_1{}^\frown w}(\epsilon,\beta)=\Tr^\circ_{\sigma_1{}^\frown w}(\delta,\beta)$, and
$\Tr^\circ_{\sigma_2{}^\frown w}(\epsilon,\alpha)=\Tr^\circ_{\sigma_2{}^\frown w}(\delta,\alpha)$.
\end{claim}
\begin{proof} As in the proof of Claim \ref{291}.
\end{proof}

Denote:
\begin{itemize}
\item $\eta_\alpha:=\max\{\sup(C^i_\tau\cap\epsilon)\mid (i,\tau)\in\Tr^\circ_{\sigma_2{}^\frown w}(\epsilon,\alpha)\}$;
\item $\eta_\beta:=\max\{\sup(C^i_\tau\cap\epsilon)\mid (i,\tau)\in\Tr^\circ_{\sigma_1{}^\frown w}(\epsilon,\beta)\}$.
\end{itemize}

By  definition of $\gamma_\beta,\eta_\beta,\zeta,w,w^*$, Claims \ref{292}, \ref{293},
and the fact that $\gamma_\beta<\epsilon<\delta$,
we get that $\eta_\beta=\max\{\gamma_\beta,\sup(C^w_\zeta\cap\epsilon)\}$.
Since $\delta\in C^w_\zeta$, we get that $\sup(C_\zeta^w\cap\epsilon)\ge\sup(e^w_\delta\cap\epsilon)>\gamma_\beta$,
and hence $\eta_\beta=\sup(C^w_\zeta\cap\epsilon)$.

Likewise, $\eta_\alpha=\sup(C^w_\zeta\cap\epsilon)$.
Denote $\eta:=\sup(C^w_\zeta\cap\epsilon)$.  Then we have just established that
\begin{itemize}
\item $\beta\in B_{\sigma_1',\eta}(\epsilon)$, for $\sigma_1':=\lift(\sigma_1{}^\frown w)\restriction \rho_{\sigma_1{}^\frown w}(\epsilon,\beta)$, and
\item $\alpha\in A_{\sigma_2',\eta}(\epsilon)$, for $\sigma_2':=\lift(\sigma_2{}^\frown w)\restriction \rho_{\sigma_2{}^\frown w}(\epsilon,\alpha)$.
\end{itemize}

By the choice of $\delta$, let us now pick an elementary submodel $M\prec\mathcal H(\chi)$
such that $\{A,B,\overrightarrow C\}\s M$ and $M\cap\lambda^+=\delta$.

The sets $B_{\sigma_1',\eta}(\epsilon)$ and $A_{\sigma_2',\eta}(\epsilon)$ are definable from parameters in $M$,
and hence belongs to $M$. Since $|M|=\lambda$ and $\beta\in B_{\sigma_1',\eta}(\epsilon)\bks M$,
we get that $|B_{\sigma'_1,\eta}(\epsilon)|=\lambda^+$.
Likewise, $|A_{\sigma_2',\eta}(\epsilon)|=\lambda^+$.

By the choice of $\epsilon$,
we have $\cf(\epsilon)>\mu_w>\cf(\zeta)$, and hence $\cf(\epsilon)>\max\{\mu_w,\cf(\zeta)\}\ge|C^w_\zeta|$.
So $\eta<\epsilon$, and hence
$\epsilon\in S^{A,B}_{>\theta}$.

Finally, recalling that $\epsilon\in e^w_\delta\cap E$, we conclude that
$S^{A,B}_{>\theta}\cap E\not=\emptyset$.
\end{proof}

\begin{thm} For every cofinal subsets $A,B$ of $\lambda^+$, and every $\theta<\lambda$,
there exists a stationary subset $S^*\s E^{\lambda^+}_{>\theta}$ such that $\rts[A\circledast B]\supseteq [S^*]^2$.
\end{thm}
\begin{proof}[Proof sketch] This is very much like the proof of Theorem \ref{t25}.
Namely, given $A,B$ and $\theta$, we appeal to Lemma \ref{34} and Fodor's lemma,
and find $\sigma_1,\sigma_2\in{}^{<\omega}\omega$ and $\eta<\lambda^+$ such that the following set is stationary:
$$S:=\{\epsilon\in E^{\lambda^+}_{>\theta}\mid |B_{\sigma_1,\eta}(\epsilon)|=|A_{\sigma_2,\eta}(\epsilon)|=\lambda^+\}.$$

We then pick a continuous $\in$-chain of relevant elementary submodels  $\langle M_i\mid i<\lambda^+\rangle$,
let $S^*:=\{ \epsilon\in S\mid M_\epsilon\cap\lambda^+=\epsilon\}$,
and argue that $\rts[A\circledast B]\supseteq [S^*]^2$.
\end{proof}

\section{Conclusions}

Building on the above and Eisworth's \cite{eisworth2},\cite{eisworth1}, we get the next theorem:

\begin{thm}\label{t31} For every singular cardinal, there exists a function $\rtts:[\lambda^+]^2\rightarrow[\lambda^+]^2\times\cf(\lambda)$
such that if $\langle a_\alpha\mid \alpha<\lambda^+\rangle$ is a family of pairwise disjoint members of $[\lambda^+]^{<\cf(\lambda)}$,
then there exists a stationary set $S\s\lambda^+$ such that for all $\langle \alpha,\beta,\delta\rangle\in [S]^2\times \cf(\lambda)$,
there exists $\alpha<\beta<\lambda^+$ such that $D[t_\alpha\times t_\beta]=\{\langle\alpha,\beta,\delta\rangle\}$.
\end{thm}
\begin{proof} In \cite{eisworth1}, Eisworth proved that there exists a function $D:[\lambda^+]\rightarrow[\lambda^+]^2\times\cf(\lambda)$
such that for every cofinal $A\s \lambda^+$, there exists a stationary subset $S\s\lambda^+$
such that $D``[A]^2\supseteq [S]^2\times\cf(\lambda)$.
Shortly afterwards, in a subsequent paper, Eisworth proved \cite{eisworth2}
that there exists a function $D':[\lambda^+]\rightarrow\lambda^+\times\lambda^+\times\cf(\lambda)$
such that if $\langle a_\alpha\mid \alpha<\lambda^+\rangle$ is a family of pairwise disjoint members of $[\lambda^+]^{<\cf(\lambda)}$,
then there are stationary subsets $S$ and $T$ of $\lambda^+$, such that for all $\langle \alpha,\beta,\delta\rangle\in S\circledast T\times \cf(\lambda)$,
there exists $\alpha<\beta<\lambda^+$ such that $D'[a_\alpha\times a_\beta]=\{\langle\alpha,\beta,\delta\rangle\}$.

Thus, given $\alpha<\beta<\lambda^+$, let
 $\rtts(\alpha,\beta):=D(\rts(\alpha',\beta'))$, where $D'(\alpha,\beta)=(\alpha',\beta',\delta')$.

 It is easy to see that $\rtts$ has the required properties.
\end{proof}
\begin{remark}
The assertion of Theorem \ref{t31} may be strengthened to require that $S$ contains ordinals of arbitrarily high cofinality.
\end{remark}

We are now in a position to prove the Introduction's Theorem \ref{t11}.

\begin{thm}\label{t32} For every singular cardinal $\lambda$, and every $\theta\le\lambda^+$, the following are equivalent:
\begin{enumerate}
\item $\lambda^+\not\rightarrow[\lambda^+]^2_\theta$;
\item $\lambda^+\not\rightarrow[\lambda^+;\lambda^+]^2_\theta$;
\item $\pr_1(\lambda^+,\lambda^+,\theta,\cf(\lambda))$.
\end{enumerate}
\end{thm}
\begin{proof}

$(1)\Rightarrow (2)$. Suppose that $f:[\lambda^+]^2\rightarrow\theta$ is a function witnessing $\lambda\not\rightarrow[\lambda^+]^2_\theta$.
Define a function $g:[\lambda^+]^2\rightarrow\theta$ by stipulating that $g(\alpha,\beta):=f(\rts(\alpha,\beta))$. Then $g$ witnesses
$\lambda\not\rightarrow[\lambda^+;\lambda^+]^2_\theta$.

$(2)\Rightarrow (3)$. Suppose that $f:[\lambda^+]^2\rightarrow\theta$ is a function witnessing $\lambda\not\rightarrow[\lambda^+;\lambda^+]^2_\theta$.
Define a function $g:[\lambda^+]^2\rightarrow\theta$ by stipulating that $g(\alpha,\beta):=f(\alpha',\beta')$,
where $\rtts(\alpha,\beta)=(\alpha',\beta',\delta)$, and $\rtts$ is the function given by the preceding theorem.
Then $g$ witnesses $\pr_1(\lambda^+,\lambda^+,\theta,\cf(\lambda))$.

$(3)\Rightarrow (1)$ Suppose that $f:[\lambda^+]^2\rightarrow\theta$ is a function witnessing $\pr_1(\lambda^+,\lambda^+,\theta,\cf(\lambda))$.
Then $f$ witnesses $\lambda\not\rightarrow[\lambda^+]^2_\theta$, as well.
\end{proof}

As an immediate corollary, we get an affirmative answer to the question raised by Eisworth in \cite{eisworth3}.
\begin{cor} If $\lambda$ is a singular cardinal and $\pr_1(\lambda^+,\lambda^+,\lambda^+,\cf(\lambda))$ fails,
then any collection of less than $\cf(\lambda)$ many stationary subsets of $\lambda^+$ must reflect simultaneously.
\end{cor}
\begin{proof} By \cite{eisworth3}, if $\lambda^+\not\rightarrow[\lambda^+]^2_{\lambda^+}$ fails,
then any collection of less than $\cf(\lambda)$ many stationary subsets of $\lambda^+$ must reflect simultaneously.
Now, appeal to Theorem \ref{t32}.
\end{proof}

\bibliographystyle{plain}

\end{document}